\documentclass[12pt,english]{smfart}

\usepackage[french,main=english]{babel}
\usepackage[utf8, latin1]{inputenc} 
\usepackage[T1]{fontenc}
\usepackage{stmaryrd,ae,euscript,enumerate,mathptmx,newtxtext,newtxmath,dsfont}
\usepackage[centering]{geometry}
\usepackage[cm]{aeguill}   

\newcommand{\inv}{^{\raisebox{.2ex}{$\scriptscriptstyle-1$}}}   

\usepackage[dvipsnames]{xcolor}   
\usepackage{xparse}
\usepackage{xr-hyper}
\usepackage[linktocpage=true,colorlinks=true,hyperindex,citecolor= Blue,linkcolor= Blue]{hyperref}   

\newtheorem{theorem}{Theorem}[section] 
\newtheorem{proposition}[theorem]{Proposition}
\newtheorem{lemma}[theorem]{Lemma}
\newtheorem{corollary}[theorem]{Corollary}
\theoremstyle{definition}

\newtheorem{remark}[theorem]{Remark} 

\begin{document}
\author{A. Goswami}
\address{
[1] Department of Mathematics and Applied Mathematics, University of Johannesburg, P.O. Box 524, Auckland Park 2006, South Africa. [2] National Institute for Theoretical and Computational Sciences (NITheCS), South Africa.}
\email{agoswami@uj.ac.za}

\title{Some classes of subsemimodule spaces}







\begin{abstract}
The aim of this paper is to study the topological properties of some classes of subsemimodules  endowed with a subbasis closed-set topology. We show that such spaces are $T_0$. When the semimodule is finitely generated, those spaces are compact as well. We characterize subsemimodule spaces for which every nonempty irreducible closed set has a unique generic point. We give a sufficient condition for a connected subsemimodule space, and using the notion of strongly disconnectedness, we determine compact disconnected subsemimodule spaces. Finally, we discuss continuous maps between subsemimodule spaces. 
\end{abstract}
\maketitle

\section{Introduction}

Since the introduction of semirings in \cite{V34}, it is natural to compare and extend results from rings to semirings. One  may think that  semirings can always to be extended to rings, but \cite{V39} gave examples of semirings which can not be embedded in rings. Moreover, the lack of additive inverses in semirings makes the semiring theory considerably different than that of rings.

As modules appear naturally in the study of rings, the notion of semimodules also inevitable in structural understanding of semirings. Some intensive study of structure of semimodules over a semiring can be found in \cite{G99, AA10, T87, CL22, H17}. Compare to this, topological study on various classes of  subsemimodules of a semimodule is less investigated.  A Zariski topology on  prime $k$-subsemimodules of a very strong multiplication semimodule has been considered in \cite{AAT11}. For a multiplication and finitely generated semimodule, in \cite{HPH21} (see also \cite{HHP22}), it has been proved that the class of prime subsemimodules endowed with Zariski topology is a spectral space.

The aim of this paper is twofold. On the one hand, our work generalizes the topological results in \cite{AAT11, HPH21, HHP22} to various classes of subsemimodule spaces; on the other hand, it extends some of the results from \cite{DG22} on ideal spaces ((\textit{i.e.}, a class of ideal of a commutative ring endowed with closed subbasis topology) to the semimodule context.
 
\section{Preliminaries}\label{prelim}
Let us recall some elementary definitions and facts about semimodules. For a detailed study of semimodules over a semiring, we refer our readers to \cite{G99}. a \emph{semiring} is a system $(R,+,0,\cdot, 1)$ such that $(R,+,0)$ is a commutative monoid, $(R, \cdot,1)$ is a monoid, $0r=0=r0$ for all $r\in R,$ and $\cdot$ distributes over $+$ both from the left and
from the right sides. In this paper all semirings are assumed to be commutative, \textit{i.e.}, $(R, \cdot,1)$ is a commutative monoid. If $R$ is semiring, then a \emph{(left) $R$-semimodule} is an additive commutative monoid $M$ together with a map (called \emph{scalar multiplication}) $R\times M\to M$, denoted by $(r,m)\mapsto rm$, such that (i) $(r+r')m=rm+r'm,$ (ii) $r(m+m')=rm+rm',$ (iii) $(rr')m=r(r'm),$ (iv) $1m=m$, and (v) $r0=0=0m$ for all $r,$ $r'\in R$ and $m, m'\in M.$ Note that we have used $0$ to denote both the additive identity of the semiring $R$ and that of the monoid $M$.

A nonempty subset $N$ of a left $R$-semimodule $M$ is called a \emph{subsemimodule} of $M$ if $N$ is closed under addition and scalar multiplication. We now define some distinguished classes of subsemimodules of a semimodule $M$. Some of these classes of subsemimodules are well-known and we provide reference(s), whereas, to best of our knowledge, others are introduced here for the first time. 

A subsemimodule $N$ is called \emph{subtractive} or \emph{k-subsemimodule} \cite{G99} if $x+y\in N$ and $x\in N$ implies $y\in N;$ whereas $N$ is called \emph{strong} \cite{G99} if $x+y\in $ implies $x,$ $y\in N$ for all $x, y\in M$. A \emph{proper} subsemimodule $N$ is such that $N\neq M.$ A proper subsemimodule $N$ is called \emph{maximal} \cite{Y10} if there is no other proper subsemimodule containing $N$. A proper
subsemimodule $N$  is called \emph{prime} \cite{Y10, YOT10} (\emph{primary}) \cite{AK10, TZ13} if $rm \in N$ with
$r \in R$ and $m \in M$ implies that $m \in N$ or $r \in (N : M)=\{r\in R\mid rM\subseteq N \}$ ($r^n \in (N : M)=\{r\in R\mid rM\subseteq N \}$ for some $n\in \mathds{N}$). A proper subsemimodule $N$ is called a \emph{weakly prime} \cite{BM14} subsemimodule if whenever $0\neq rm \in N,$ for some $r \in R$ and $m \in M,$ then either $m \in N$ or
$rM\subseteq N.$ An element $r$ of $R$ is called \emph{prime to $N$} if $rm \in N$ with
$r \in R$ and $m \in M$ implies that $m \in N$. Denote the set of all prime to $N$ elements of $R$ by $\mathrm{Prm}(N).$ A proper subsemimodule $N$ is called \emph{primal} \cite{BM14} if $R\setminus \mathrm{Prm}(N)$ is an ideal of $R$. A \emph{semiprime} subsemimodule $N$ is the intersection of prime semimodules of $M$. For an alternative notion of a semiprime semimodule, see \cite{TZ13}.  The \emph{radical} \cite{AK10} $\sqrt{N}$ of a subsemimodule $N$ is the intersection fo all prime subsemimodules of $M$ containing $N$. A prime subsemimodule $N$ is called \emph{extraordinary} \cite{HPH21} if for any semiprime subsemimodules $L$ and $K$ of $M,$ $L \cap K \subseteq  N$
implies that $L \subseteq  N$ or $K \subseteq N.$ 

A subsemimodule $N$ is called \emph{strongly irreducible} if for any two subsemimodules $L$ and $K$ of $M$, $L\cap K\subseteq N$ implies $L\subseteq N$ or $K\subseteq N.$  A subsemimodule $N$  that is not the intersection of two (any set of) subsemimodules of $M$ in which $N$ is properly contained is called \emph{irreducible} (\emph{completely irreducible}). If a subsemimodule $N$ is generated by a single  element (a finite set of elements) of $M$, then $N$ is called  \emph{cyclic} (\emph{finitely generated}). A nonzero subsemimodule $N$ is called a \emph{minimal} subsemimodule if $N$ contains no other nonzero subsemimodules of $M.$ A \emph{minimal prime} subsemimodule $N$ is both a minimal and a prime subsemimodule. To denote any one of the above-mentioned distinguished class of subsemimodules of $M$, we use the notation $\mathrm{Dis}_M$ and we assume that $M\notin \mathrm{Dis}_M$; whereas, $\mathrm{Sub}_M$ denotes the set of all subsemimodules of $M$.

\section{Subsemimodule spaces}\label{ssms}

Let $M$ be an $R$-semimodule. For any subset $S$ of $M$ and for any $\mathrm{Dis}_M$, define
\begin{equation}\label{vsl}
\mathcal{V}(S)=\{L\in \mathrm{Dis}_M\mid S\subseteq L\}.	
\end{equation} 

\begin{lemma}\label{bpvs}
For a given $R$-semimodule $M$, the subsets of $\mathrm{Dis}_M$ determined by $\mathcal{V}$ have the following properties.
\begin{enumerate}[\upshape (1)]
	
\item $\mathcal{V}(S)\supseteq \mathcal{V}(S')$ for any two nonempty subsets $S$ and $S'$ of $M$ such that $S\subseteq S'$.
	
\item\label{vsvs} $\mathcal{V}(S)=\mathcal{V}(\langle S \rangle)$ for any nonempty subset $S$ of $M$.	
	
\item  \label{a} $\mathcal{V}(0)=\mathrm{Dis}_M$ and $\mathcal{V}(M)=\emptyset.$

\item \label{b} $\bigcap_{\lambda\in \Lambda}\mathcal{V}(N_{\lambda})=\mathcal{V}\left( \sum_{\lambda \in \Lambda} N_{\lambda}\right)$ for all $N_{\lambda}\in \mathrm{Sub}_M$ and for all $\lambda\in \Lambda$.

\item\label{vnuvn} $\mathcal{V}(N)\cup \mathcal{V}(N') \subseteq \mathcal{V}(N\cap N')$ for all $N,$ $N'\in \mathrm{Sub}_M.$

\item\label{vnsvn} $\mathcal{V}(N)\supseteq \mathcal{V}(\sqrt{N})$ for any subsemimodule $N$ of $M$.

\item\label{axaa} Let $N$ be subsemimodule of $M$. Define $ N^{\omega}:=\bigcap\{L\mid L\in \mathcal{V}(N)\}$. Then $N\subseteq N^{\omega}$.

\item \label{axa}   
$N=N^{\omega}$, whenever $N\in\mathrm{Dis}_M.$

\item \label{hahxa}  $\mathcal{V}(N)=\mathcal{V}(N^{\omega}),$ for all subsemimodule $N$ of $M$.

\item\label{xbxa} For any two subsemimodules $N_1$ and $N_2$ of $M$, $\mathcal{V}(N_1)\subseteq \mathcal{V}(N_2)$ if and only if $N_2^{\omega}\subseteq N_1^{\omega}$.
\end{enumerate}
\end{lemma}

\begin{proof}
Straightforward.
\end{proof}

Note that we have equality in Lemma \ref{bpvs}(\ref{vnsvn}) whenever $\mathrm{Dis}_M$ is the class of prime, minimal prime, or maximal subsemimodules.
Our aim is to construct a topology on a $\mathrm{Dis}_M$ using the subsets of the form $\{\mathcal{V}(S)\},$ and instead of (\ref{vsl}), thanks to Lemma \ref{bpvs}(\ref{vsvs}), it is now sufficient to consider only the subsets
\begin{equation}
\mathcal{V}(N)=\{L\in \mathrm{Dis}_M\mid N\subseteq L\}\qquad (N\in \mathrm{Sub}_M).
\end{equation}
From Lemma \ref{bpvs}(\ref{a}) it follows \[\bigcap_{N\in \mathrm{Sub}_M} \mathcal{V}(N)=\emptyset,\] and by \cite[Theorem\,15\,A.13., p.\,254]{C66}, this implies that the collection $\{\mathcal{V}(N)\}_{N\in \mathrm{Sub}_M}$ as  a closed subbasis generates a unique topology on  $\mathrm{Dis}_M$. We denote this topology by $\tau_{\mathrm{Dis}_M}$. 

In general, the sets $\mathcal{V}(N)$ are not closed sets but only subbasis closed sets. This is due the fact that for an arbitrary $\mathrm{Dis}_M$, we do not have equality in Lemma \ref{bpvs}(\ref{vnuvn}) (\textit{i.e.}, sets $\mathcal{V}(N)$ are not closed under finite union). It has been shown in \cite[Lemma 3.2]{HPH21} that equality holds in Lemma \ref{bpvs}(\ref{vnuvn}) if and only if $M$ is a top semimodule, and in that case, our closed-subasis topology coincides with the Zariski topology endowed on any $\mathrm{Dis}_M$. With an abuse of notation, we simply write $\mathrm{Dis}_M$ to denote the topological space $(\mathrm{Dis}_M, \tau_{\mathrm{Dis}_M})$. 

For compactness of a subsemimodule space $\mathrm{Dis}_M$, we require the property:
\begin{equation}\label{scc}
\mathcal{V}(N)=\emptyset \implies N=M.
\end{equation}
In general, (\ref{scc}) is not true for an arbitrary class of subsemimodules. However, if $M$ is finitely generated, then (\ref{scc}) holds for any $\mathrm{Dis}_M$. For any semimodule, the following result gives a necessary condition for a class of subsemimodules have the property (\ref{scc}). 

\begin{lemma}\label{ncc}
If a subsemimodule space $\mathrm{Dis}_M$ has the property $(\ref{scc})$, then $\mathrm{Dis}_M$ contains all maximal subsemimodules of $M$.
\end{lemma}

\begin{proof} Let $\mathrm{Dis}_M$ has the property (\ref{scc}). Let $N$ be a maximal subsemimodule of $M$ such that $N\notin \mathrm{Dis}_M$. This implies $\mathcal{V}(N)=\emptyset$ and hence $N=M$, a contradiction.
\end{proof} 

If $M$ is finitely generated, then the converse of the above Lemma \ref{ncc} is also true. Since our topology is generated by subbasis closed sets, in the proof of the following theorem we need to rely on Alexander's subbasis theorem.

\begin{theorem}\label{comp} 
If $M$ is a finitely generated $R$-semimodule, then any subsemimodule space $\mathrm{Dis}_M$ is compact. 
\end{theorem} 

\begin{proof}   
Suppose $\bigcap_{\lambda\in \Lambda}K_{ \lambda}=\emptyset$ for a family  $\{K_{ \lambda}\}_{\lambda \in \Lambda}$ of subbasis closed sets of a subsemimodule space $\mathrm{Dis}_M$. By Alexander's subbasis theorem, it suffices to show that $\bigcap_{ i\,=1}^{ n}K_{ \lambda_i}=\emptyset,$ for some finite subset $\{\lambda_1, \ldots, \lambda_n\}$ of $\Lambda$. Let $\{N_{ \lambda}\}_{\lambda \in \Lambda}$ be a class of subsemimodules of $M$ such  that $K_{ \lambda}=\mathcal{V}(N_{ \lambda})$ for all $\lambda \in \Lambda,$  By Lemma \ref{bpvs}(\ref{b}),   $\mathcal{V}\left(\sum_{\lambda \in \Lambda}N_{ \lambda}\right)=\emptyset,$ and by property  (\ref{scc}), this implies $ \sum_{\lambda \in \Lambda}N_{ \lambda}=M.$ Since $M$ is finitely generated, we have the desired finite subset of $\Lambda$.
\end{proof}  

By \cite{P94}),  a topological space $X$ is $T_0$ if and only if specialization order of $X$ is a partial order, and this implies

\begin{proposition}
Every $\mathrm{Dis}_M$ is  a $T_0$-space. 
\label{ct0t1}  
\end{proposition}

\begin{corollary}\cite[Theorem 3.1]{HPH21}. If $M$ is a top $R$-semimodule, then the spectrum of prime subsemimodules of $M$ endowed with Zariski topology is a $T_0$-space.
\end{corollary}

We recall a few topological concepts that we are going to use next. If $X$ is a topological space, a closed subset $S$ is called \emph{irreducible} if $S$ is not the union of two properly smaller closed subsets $S_{1}, S_{ 2}\subsetneq S.$ A point $x$ in a closed subset $S$ is called a \emph{generic point} of $S$ if $S = \mathcal{C}(\{x\}),$ the closure of $\{x\}$. A $R$-semimodule $M$ is said to be \emph{Noetherian} if every ascending chain of subsemimodules of $M$ is eventually
stationary. A topological space $X$ is said to be \emph{Noetherian} if the closed subsets of $X$ satisfy the descending chain condition. 

Suppose $M$ is a top $R$-semimodule and $\mathrm{Dis}_M$ is the class of prime subsemimodules of $M$ endowed with the Zariski topology. Then according to \cite[Theorem 3.3]{HPH21}, a nonempty subset $Y$ of $\mathrm{Dis}_M$ is irreducible if and only if  the intersection of all prime subsemimodules belonging to $Y$ is itself a prime subsemimodule of $M$. For a $R$-semimodule $M$ and for an arbitrary subsemimodule space we do not a similar characterization, however, in Lemma \ref{irrc} we identify a class of irreducible closed subsets which is good enough to characterize $T_1$ subsemimodule spaces. Moreover, this lemma is going to be fruitful in studying connectedness of subsemimodule spaces.

\begin{lemma}\label{irrc}
The subbasis closed sets of the form $\{\mathcal{V}(N)\mid N\in \mathcal{V}(N)\}$ of a subsemimodule space are irreducible. 
\end{lemma} 

\begin{proof} 
Let $\mathcal{C}(N)$ be the closure of the subsemimodule $N$. It is sufficient to show that $\mathcal{V}(N)=\mathcal{C}(N)$ for every $N\in\mathcal{V}(N)$. From the definition of  $\mathcal{C}(N)$, it is obvious that $\mathcal{C}(N)\subseteq \mathcal{V}(N)$. 
To have the other inclusion, consider the case: $\mathcal{C}(N)= \mathrm{Dis}_M$. Then  from 
$
\mathrm{Dis}_M=\mathcal{C}(N)\subseteq \mathcal{V}(N)\subseteq \mathrm{Dis}_M,
$
it follows that $\mathcal{V}(N)=\mathcal{C}(N)$. 

Let $\mathcal{C}(N)\neq \mathrm{Dis}_M$. Since $\mathcal{C}(N)$ is a closed set, there exists an $\Lambda$ such that  for each $\lambda\in\Lambda$, there is a positive integer $n_{\lambda}$ and subsemimodules $N_{\lambda 1},\dots, N_{\lambda n_\lambda}$ of $M$ such that 
\[
\mathcal{C}(N)={\bigcap_{\lambda\in\Lambda}}\left({\bigcup_{ i\,=1}^{ n_\lambda}}\mathcal{V}(N_{\lambda i})\right).
\]
Since
$\mathcal{C}(N)\neq \mathrm{Dis}_M,$  for each $\lambda$, ${\bigcup_{ i\,=1}^{ n_\lambda}}\mathcal{V}(N_{\lambda i})\neq \emptyset$. Hence, for each $\lambda$, $N\in   {\bigcup_{ i\,=1}^{ n_\lambda}}\mathcal{V}(N_{\lambda i})$, and that implies \[\mathcal{V}(N)\subseteq {\bigcup_{ i=1}^{ n_\lambda}}\mathcal{V}(N_{\lambda i}),\] which proves the inclusion.	
\end{proof}     

\begin{corollary}\label{spiir}
All non-empty subbasis closed subset of proper subsemimodule spaces are irreducible.
\end{corollary} 

It has been shown in \cite[Theorem 3.2]{HPH21} that in a top semimodule, the class of prime subsemimodules endowed with Zariski topology is a $T_1$-space if and only if every prime subsemimodule is maximal in the class of prime subsemimodules. Theorem \ref{t1} generalizes this result to all $\mathrm{Dis}_M$ spaces.

\begin{theorem}\label{t1}
A $\mathrm{Dis}_M$ is a $T_1$-space if and only if it is contained in the class of maximal subsemimodule space. 
\end{theorem}  

\begin{proof} 
If $L\in\mathrm{Dis}_M$, then obviously $L\in\mathcal{V}(L)$, and so, by Lemma \ref{irrc}, $\mathcal{C}(L)=\mathcal{V}(L)$. Let $N$ be a maximal subsemimodule of $M$ with $L\subseteq N$. If $\mathrm{Dis}_M$ is a $T_1$-space, then   \[N\in 	\mathcal{V}(L)=\mathcal{C}(L) = \{L\}.\] Therefore, $N=L$.  
Conversely,  $\mathcal{V}(N)=\{N\}$ for every maximal subsemimodule $N$ of $M$ so that $N\in \mathcal{V}(N)$, and hence, by Lemma \ref{irrc}, $\mathcal{C}(N)=\{N\}$. This proves that the class of maximal subsemimodules of $M$ is a $T_1$-space.
\end{proof} 

\begin{corollary}
If $M$  is a Noetherian semimodule and if $\mathrm{Dis}_M$ is a discrete space, then $M$ is Artinian.  
\end{corollary}

Determining whether every nonempty irreducible closed set of a topological space has a unique generic point is a fundamental question in topology. It has been proved in \cite[Corollary 3.1]{HPH21} that the class of prime subsemimodules of a top semimodule endowed with Zariski topology has this property. In Theorem \ref{sob} we characterize such subsemimodule spaces.

\begin{theorem}\label{sob}  
Let $\mathrm{Dis}_M$ be a subsemimodule space. Then every non-empty irreducible subbasis closed set $\mathcal{V}(N)$ of $\mathrm{Dis}_M$ has a unique generic point if and only if $\mathcal{V}(N)$ contains $N^{\omega}.$ 
\end{theorem} 

\begin{proof}
Let $\mathcal{V}(N)$ is a non-empty irreducible subbasis closed subset of a subsemimodule space $\mathrm{Dis}_M$ such that $\mathcal{V}(N)$ has a unique generic point, which implies that 
\[\mathcal{V}(N)=\mathcal{C}(L)=\mathcal{V}(L).\] Then by Lemma \ref{bpvs}(\ref{axa}) we have $L=N^{\omega}\in \mathrm{Dis}_M$. 
The proof of the converse part is more involved. Assume that every non-empty irreducible subbasis  closed set $\mathcal{V}(N)$ contains $N^{\omega}$ and let $K$ be an irreducible closed subset of $\mathrm{Dis}_M$. By definition, \[K=\bigcap_{i\in I}\left( \bigcup_{\alpha=1}^{n_i} \mathcal{V}(N_{i\alpha})\right),\] for some  subsemimodules $N_{ij}$ of $M$. Since $K$ is irreducible, for every $i\in I$ there exists a subsemimodule $N_i$ of $M$ such that $K\subseteq \mathcal{V}(N_i)\subseteq \bigcup_{\alpha=1}^{n_i} \mathcal{V}(N_{i\alpha})$ and thus we have  $$K=\bigcap_{i\in I}\mathcal{V}(N_i)=\mathcal{V}\left(\sum_{i\in I}N_i \right)=\mathcal{V}\left(\sum_{i\in I}N_i \right)^{\omega}.$$ Since $\left(\sum_{i\in I}N_i \right)^{\omega}\in \mathrm{Dis}_M$, we thus have $K=\mathcal{C}\left(\sum_{i\in I}N_i \right)^{\omega}$.  
\end{proof}

\begin{corollary}
Every nonempty irreducible closed subset of a proper, prime, minimal prime subsemimodule spaces has a unique generic point.
\end{corollary}

\begin{corollary}
Let $M$ be a top $R$-semimodule. Then every nonempty irreducible closed subset (with respect to Zariski topology) of a prime and minimal prime subsemimodule spaces has a unique generic point.	
\end{corollary}

\begin{remark}
Recall that a topological space $X$ is called \emph{sober} if every nonempty irreducible closed subset of $X$ has a unique generic point. Sobriety is a separation axiom in between $T_0$ and $T_2$ (and independent of $T_1$). Theorem \ref{sob} in fact determines subsemimodule sober spaces.
\end{remark}

\begin{proposition}\label{sirrs} 
Every nonempty irreducible closed subset of a strongly irreducible subsemimodule space has a unique generic point.
\end{proposition}

\begin{proof} 
Since $\mathrm{Des}_M$ is a multiplicative lattice with product as intersection, every strongly irreducible subsemimodule becomes a prime element of this lattice and hence by \cite[Lemma 2.6]{FFJ22},  we have the desired conclusion.
\end{proof}

The usual notion of connectedness does not have a direct formulation in terms of subbasis closed sets. We, therefore, introduce a different kind of connectedness (in terms of subbasis closed sets) of subsemimodule spaces that is related to the conventional notion of connectedness.
We say a closed subbasis $\mathcal{S}$ of a topological space $X$ \emph{strongly disconnects} $X$  if there exist two non-empty subsets $A,$ $B\in\mathcal{S}$ such that $X=A\cup B$ and $A\cap B=\emptyset$.

From the above definition, it is clear that if some closed subbasis strongly disconnects a space, then the space is disconnected. Moreover, if a space is disconnected, then some closed subbasis strongly disconnects it. But this does not imply that every closed subbasis strongly disconnects the space. However, we have the following result from \cite{DG22}.

\begin{lemma}\label{th1}
If a compact space is disconnected, then every closed basis that is closed under finite intersections strongly disconnects the space. Moreover, for any subsemimodule space $\mathrm{Dis}_M$, the closed basis is closed under binary intersections.
\end{lemma}

Based on the above Lemma \ref{th1}, we immediately have the following result.

\begin{theorem}\label{cor1}
Suppose $\mathrm{Dis}_M$ is a compact space. Then  $\mathrm{Dis}_M$ is  disconnected if and only if the closed basis strongly disconnects $\mathrm{Dis}_M$. 
\end{theorem}   

A consequence of Lemma \ref{irrc} is the following sufficient condition for a subsemimodule space to be connected. 

\begin{theorem}\label{conis}
If  $0\in\mathrm{Dis}_M,$  then the subsemimodule space $\mathrm{Dis}_M$ is connected.
\end{theorem}  

\begin{proof}
Since by Lemma \ref{bpvs}(\ref{a}), $\mathrm{Dis}_M=\mathcal{V}(0)$ and since irreducibility implies connectedness, the claim follows from Lemma \ref{irrc}. 
\end{proof}

\begin{corollary}
Proper, finitely generated, cyclic subsemimodule spaces are connected. 
\end{corollary} 

Suppose $M\to M'$ is an $R$-semimodule homomorphism. Likewise rings, it is natural to expect a continuous map  $\phi\colon \mathrm{Dis}_{M'}\to \mathrm{Dis}_M$ between similar type subsemimodule spaces. It is indeed so, and we shall discuss that in the next proposition. The only issue here is that the inverse image of a particular type of subsemimodule might not be of the same type. So, we need to consider only those $R$-semimodule homomorphisms for which such preservation holds. We call this a \emph{contraction property}. Furthermore, because of the nature of our topology, instead of closed sets, in the arguments of the proofs, we shall require subbasis closed sets. 

\begin{proposition}\label{conmap}
Let $\mathrm{Dis}_M$ be a subsemimodule space and $\phi\colon M\to M'$ be an $R$-semimodule homomorphism having the contraction property with respect to $\mathrm{Dis}_M$.   Let $N'\in\mathrm{Dis}_{M'}.$ Define a map $\phi^*\colon  \mathrm{Dis}_{M'}\to \mathrm{Dis}_{M}$ by  $\phi^*(N)=\phi\inv(N)$. Then
\begin{enumerate}[\upshape (1)]
		
\item \label{contxr} the map $\phi^*$ is    continuous.
		
\item \label{shcs} If $\phi$ is  surjective, then the subsemimodule space $\mathrm{Dis}_{M'}$ is homeomorphic to the closed subspace $\mathcal{V}(\mathrm{ker}\,\phi)$ of the subsemimodule space $\mathrm{Dis}_{M}.$
		
\item \label{imde} The image  $\phi^*(\mathrm{Dis}_{M'})$ is dense in $\mathrm{Dis}_{M}$ if and only if $$\mathrm{ker}\,\phi\subseteq \bigcap_{L\in \mathrm{Dis}_{M}}L;$$
\end{enumerate}
\end{proposition}

\begin{proof}      
To show (\ref{contxr}), let $\mathcal{V}(N)$ be a   subbasic closed set of the subsemimodule  space $\mathrm{Dis}_M.$ Observe  that $$(\phi^*)\inv(\mathcal{V}(N)) =\{N'\in \mathrm{Dis}_{M'}\mid \phi(N)\subseteq N'\}=\mathcal{V}(\langle \phi(N)\rangle),$$ and hence the map $\phi^*$  continuous.     
For the homeomorphism in (\ref{shcs}), observe that     
$\mathrm{ker}\,\phi\subseteq \phi\inv(N'),$ in other words, $\phi^*(N')\in \mathcal{V}(\mathrm{ker}\,\phi)$. This implies that $\mathrm{im}\,\phi^*=\mathcal{V}(\mathrm{ker}\,\phi).$  
Since for all $N'\in \mathrm{Dis}_{M'},$ \[\phi(\phi^*(N'))=N'\cap \mathrm{im}\,\phi=N',\] the map $\phi^*$ is injective. To show that $\phi^*$ is a closed map, first we observe that for any subbasis closed subset  $\mathcal{V}(N)$ of  $\mathrm{Dis}_{M'}$, we have \[\phi^*(\mathcal{V}(N))=\phi\inv\{ L'\in \mathrm{Dis}_{M'}\mid N\subseteq   L'\}=\mathcal{V}(\phi\inv(N)).\]   Now if $K$ is a closed subset of $\mathrm{Dis}_{M'}$ and if \[K=\bigcap_{ \alpha \in \Omega} \left(\bigcup_{ i \,= 1}^{ n_{\alpha}} \mathcal{V}(N_{ i\alpha})\right),\] then \[\phi^*(K)=\phi\inv \left(\bigcap_{ \alpha \in \Omega} \left(\bigcup_{ i = 1}^{ n_{\alpha}} \mathcal{V}(N_{ i\alpha})\right)\right)=\bigcap_{ \alpha \in \Omega} \bigcup_{ i = 1}^{ n_{\alpha}} \mathcal{V}(\phi\inv(N_{ i\alpha})),\] a closed subset of  $\mathrm{Dis}_M.$ Since $\phi^*$ is   continuous, we have the proof.
To prove (\ref{imde}), we first show that $\mathcal{C}(\phi^*(\mathcal{V}(N')))=\mathcal{V}(\phi\inv(N')),$ for all subsemimodules $N'\in \mathrm{Sub}_{M'}.$ To this end, let $L\in \phi^*(\mathcal{V}(N')).$ This implies $\phi(L)\in \mathcal{V}(N'),$ which means $N'\subseteq \phi(L).$ In other words, $L\in \mathcal{V}(\phi\inv(N')).$ The other inclusion follows from the fact that $\phi\inv(\mathcal{V}(N'))=\mathcal{V}(\phi\inv(N')).$ Since $$\mathcal{C}(\phi^*(\mathrm{Dis}_{M'}))=\mathcal{V}(\phi\inv(\mathfrak{o}))=\mathcal{V}(\mathrm{ker}\,\phi),$$ the closed subspace $\mathcal{V}(\mathrm{ker}\,\phi)$ is equal to $\mathrm{Dis}_M$ if and only if $\mathrm{ker}\,\phi\subseteq \cap_{L\in \mathrm{Dis}_M}L.$     
\end{proof}    

\begin{corollary}
The subsemimodule space $\mathrm{Dis}_{M/N}$ is homeomorphic to the closed subspace
$\mathcal{V}(N)$ of $\mathrm{Dis}_M$.   
\end{corollary}

It has been shown in \cite{HPH21} that the class of prime subsemimodules of a finitely generated top semimodule is spectral in the sense of \cite{H69}. An attempt has been made in \cite{FGS22} to classify spectral spaces among the ideal spaces of a commutative ring. It would be interesting to find out which subsemimodule spaces are spectral. It is clear that one cannot just copy the arguments from \cite{FGS22} to determine that. It will possibly require some additional methods.


\begin{thebibliography}{1}
	
\bibitem{AA10}
\textsc{R. E. Atani} and \textsc{S. E. Atani}, On subsemimodules of semimodules, \textit{Bull. Acad. Stiinte.
Repub. Mold Mat.}, 63 (2010), 20--30.

\bibitem{AAT11}
\textsc{S. E. Atani}, \textsc{R. E. Atani}, and \textsc{U. Tekir}, A Zariski topology for semimodules, \textit{Eur. J. Pure Appl. Math.},
4(3) (2011), 251--265.

\bibitem{AK10}
\textsc{E. Atani} and \textsc{A. Kohan}, A note on finitely generated multiplication semimodules over commutative semirings, \textit{Int. J. Algebra}, 4(5-8) (2010), 389--396.

\bibitem{BM14}
\textsc{M. Bataineh} and
\textsc{R. Malas},  Primal and weakly primal sub semi modules, \textit{Amer. Int. J. Contem. Research}, 4(1) (2014), 131--135.
  
\bibitem{C66}
E. \v{C}ech,  \emph{Topological spaces}, Int. Publ. John Wiley \& Sons., 1966.

\bibitem{CL22}
\textsc{I. Chajdaa} and \textsc{H. L\"{a}nger}, Semimodules over commutative semirings and modules over
unitary commutative rings, \textit{Linear multilinear algebra}, 70(7) (2022), 1329--1344.

\bibitem{DG22} T. Dube and A. Goswami, Ideal spaces: An extension of structure spaces of rings, J. Alg. Appl. (online ready). 
 
\bibitem{FGS22} \textsc{C. A. Finocchiaro}, \textsc{A. Goswami}, and \textsc{D. Spirito}, Distinguished classes of ideal spaces and their topological properties, \textit{Comm. Alg}., 2022 (online ready).

\bibitem{FFJ22} \textsc{A. Facchini}, \textsc{C. A. Finocchiaro}, and
\textsc{G. Janelidze}, Abstractly constructed prime spectra, \textit{Algebra Univers.}, 83(8) (2022), 1--38.
 
\bibitem{G99} 
\textsc{J. S. Golan}, \textit{Semirings and their applications}, Springer, 1999.

\bibitem{H17}  
\textsc{S. C. Han}, Maximal and prime k-subsemimodules in semimodules over semirings, \textit{J. Algebra Appl.},
16(7) (2017) 1750130.

\bibitem{H69} 
\textsc{M. Hochster}, Prime ideal structure in commutative rings, \textit{Trans. Am. Math. Soc.}, 142 (1969), 43--60.

\bibitem{HPH21} \textsc{C-H Han}, \textsc{W-S Pae}, and \textsc{J-N Ho}, Topological properties of the prime spectrum of
a semimodule, \textit{J. Algebra}, 566 (2021) 205--221.

\bibitem{HHP22}
\textsc{C-H Han}, \textsc{J-N Ho}, and \textsc{W-S Pae}
Properties of the subtractive prime spectrum of a
semimodule, Hacet. J. Math. Stat.
XX(x) (2022), 1--14.

\bibitem{P94} \textsc{H. A. Priestley}, Intrinsic spectral topologies, \emph{Papers on general topology and applications},
(Flushing, NY, 1992), 78--95, Ann. New York Acad. Sci., 728, New York Acad. Sci., New York, 1994.

\bibitem{T87}
\textsc{M. Takahashi}, Structures of semimodules. \textit{Kobe J Math.}, 4 (1987), 79--101.
 
\bibitem{TZ13}
\textsc{H. A. Tavallaee} and \textsc{M. Zolfaghari}, On semiprime subsemimodules and related results, \textit{J. Indones. Math. Soc.},
19(1) (2013), 49--59.

\bibitem{V34}
\textsc{H. S. Vandiver}, Note on a simple type of algebra in which the cancellation law
of addition does not hold, \textit{Bull. Am. Math. Soc.}, 40(12) (1934), 914--920.

\bibitem{V39}
---------, On some simple types of semi-rings, \textit{Am. Math. Monthly}, 46(1) (1939), 22--26.

\bibitem{Y10} 
\textsc{G. Ye\c silot}, On prime and maximal k-subsemimodules of semimodules, \textit{Hacet. J. Math.
Stat.}, 39(3) (2010) 305--312.

\bibitem{YOT10}
\textsc{G. Ye\c silot}, \textsc{K. Orel}, and \textsc{U. Tekir}, On prime subsemimodules, \textit{Int. J. Algebra}, 4 (2010) 53--60.
\end{thebibliography}
\end{document}